\documentclass[10pt]{article}
\usepackage{euscript,amsfonts,amssymb,amsmath,amscd,amsthm}
\usepackage[breaklinks=true]{hyperref}
\usepackage[english]{babel}
\theoremstyle{plain}
\newtheorem{theorem}{Theorem}
\newtheorem*{lemma}{Lemma}
\newtheorem*{cor}{Corollary}
\theoremstyle{remark}
\newtheorem*{rem}{Remark}
\newtheorem{remark}{Remark}
\newtheorem*{ex}{Example}
\renewcommand{\a}{\alpha}
\newcommand{\A}{A}
\renewcommand{\ge}{\geqslant}
\renewcommand{\le}{\leqslant}
\renewcommand{\d}{\bar{\textup{d}}}
\DeclareMathOperator*{\dlim}{d-lim}
\title{On the asymptotic behavior of cocycles over flows}
\author{Maxim E. Lipatov}
\date{}
\begin{document}
\maketitle
\begin{abstract}
In 1968, V.I. Oseledets formulated the question of convergence in the Birkhoff theorem and the multiplicative ergodic theorem for measurable cocycles over flows under the condition of integrability for each individual~$t$. A.M. Stepin and the author established (2016) the convergence along subsets of time of density 1. In this note, we show that moreover the convergence is fulfilled modulo time subsets of finite measure.
\medskip


\end{abstract}

\section{Introduction}
Let $\{T^t\}$ be a measure-preserving measurable flow on a Lebesgue space $(X,\mu)$ with $\mu(X) = 1.$
A {\it cocycle} over the flow $\{T^t\}$ with values in a group~$G$ is a measurable function $\alpha \colon\mathbb R\times X \to G$ such that $$\alpha (t + s, x) = \alpha(s, T^t x)\alpha(t, x)$$ for all $t,s\in\mathbb R$ and $x\in X.$\footnote{Similarly for semiflows.} For $G = \mathbb R,$ we have an additive cocycle: $$ \alpha (t + s, x) = \alpha (t, x) + \alpha (s, T^t x);$$ and if $\a$ is absolutely continuous with respect to $t,$ then $\a(t, x) = \int_0^t f (T^t x) \, dt $ for some measurable function $f$.\footnote {As a function $f,$ one can take $f(x) = \varlimsup \limits_{n \to \infty} \frac{\a (\varepsilon_n, x)}{\varepsilon_n}$ for some sequence $\varepsilon_n \to0. $}
The Oseledets multiplicative ergodic theorem (MET) \cite{O} generalizes the statement about convergence of the means $\a (t, x)/t $ of such cocycles (the Birkhoff theorem) to the non-commutative case. According to MET, for an arbitrary measurable cocycle $\A\colon\mathbb R \times X \to GL (m, \mathbb R)$ satisfying the condition
\begin{equation}
\sup\limits_{0\le t \le1} \ln ^ + \|A (t, x) ^ {\pm1} \|\in L^1 (X, \mu), \label{int}
\end{equation}
almost all points with respect to the invariant measure $\mu$ are Lyapunov regular. This implies that almost everywhere there are exact Lyapunov exponents as well as the block structure of the cocycle: the vector bundle $X\times\mathbb R^m$ is decomposed into a direct sum of invariant subbundles corresponding to distinct Lyapunov exponents. If $X$ is a smooth compact manifold and $\{T^t\}$ is a flow of class $C^1$ preserving the smooth measure $\mu$ on $X,$ up to a set of measure zero the tangent bundle $TX$ admits trivialization and the differential of the flow $(t, x)\mapsto A(t, x):=D_xT^t$ is a cocycle satisfying condition \eqref{int}. This explains the great importance of MET for the theory of dynamical systems, especially for nonuniformly hyperbolic theory (see \cite{BP}). MET was generalized for local fields \cite{R}, for Hilbert \cite{Ru82, GM} and Banach \cite{M, T} spaces, spaces of non-positive curvature \cite{K87, KM}.
\par In the one-dimensional case, a sufficient condition for convergence of the means of an additive cocycle $\a(t, x)$, similar to \eqref{int}, has the form
\begin {equation}
\sup\limits_{0 \le t\le1}|\a (t, x)|\in L^1 (X, \mu). \label{int2}
\end{equation}
In the work \cite{O}, V.I. Oseledets posed the question about convergence in the Birkhoff theorem and MET under the condition of integrability for each individual $t.$
In the joint work of the author and A.M. Stepin \cite{LS}, it was shown that although convergence in all $t$ under these conditions may not hold, it does along density 1 time subsets depending on $x \in X.$ \footnote{There was used the argument with the function $ \varphi(x) = x^{-1/2} $ on page \pageref{phi} below.} In particular, there exist exact Lyapunov exponents for cocycles in the sense of the specified convergence almost everywhere. Such generalized Lyapunov exponents seem natural from an applied point of view since they are not sensitive to rare outliers.

Here we prove a stronger statement: the convergence persists if a set of finite Lebesgue measure (depending on x) is discarded from the time axis.

\section {Birkhoff theorem}

Here are some examples that demonstrate the possible behavior of the Birkhoff means of additive cocycles. We first recall the construction of a suspension flow over an automorphism $S$ of a Lebesgue space $(Y,\nu)$ with a measurable roof function $f\colon Y\to\mathbb R_{+},$ $\int f d\nu=1,$ $f(y)\ge C>0.$ Such a flow $\{T^t\}$ acts on the space $X=\{(y,\tau)\in Y\times\mathbb R:0\le\tau<f(y)\}$ with the measure $d\mu=d\nu\,dt$ by the formula $$T^t(y,\tau)=\begin{cases}
(y,\tau+t),&0\le\tau+t<f(y),\\
(S^ny,\tau+t-f_n(y)),&f_n(y)\le\tau+t<f_{n+1}(y),
\end{cases}$$
where $f_n(y):=\sum\limits_{i=0}^{n-1}f(S^iy).$ In this case, one can always go to an isomorphic flow with $C\le f(y) \le 2C$ \cite{Ro}.

\begin{ex} Let $\{T^t\}$ be an arbitrary flow on $(X,\mu)$ without fixed points. (At the fixed points the means $\a(t, x)/t$ obviously converge.) Consider a coboundary, i.e., a cocycle of the form $\a(t,x)=h(T^tx)-h(x)$ with a function $h\in L^1,$ whose values along the trajectory $h(T^tx)$ are unbounded on each of consecutive time intervals.
The existence of such a function follows from the suspension representation of the flow. The flow under consideration is isomorphic to at most a countable sum of suspension flows with decreasing $C$. For each such suspension, we define a function on its space unbounded on each fiber $\{y\}\times[0, f(y))$. This gives an example of an integrable cocycle whose means do not converge everywhere.
\end{ex}

In the same way, one can construct a similar example of a cocycle not cohomologous to 0 using the following statement.

\begin{lemma} Any cocycle $\alpha\colon\mathbb R\times X\to \mathbb R$ over a suspension flow $\{T^t\}$ is uniquely determined by the values $\alpha(t,(y,0)),$ $0\le t\le f(y),$ $y\in Y$.
\end{lemma}

\begin{proof} Indeed, if $\alpha\colon\mathbb R\times X\to \mathbb R$ is a cocycle, then for $0\le\tau<f(y)$ and $f_n(y)\le\tau+t<f_{n+1}(y)$
we have $$\a(t,(y,\tau))=\a(\tau+t,(y,0))-\a(\tau,(y,0))=\sum_{i=0}^{n-1}\a(f(S^iy),(S^iy,0))+$$
$$+\a\bigl(\tau+t-f_n(y), (S^ny,0)\bigr)-\a(\tau,(y,0)).$$
It is easy to verify that the function $\a$ given by this formula is a cocycle.
\end{proof}

\begin{rem} For ergodic $S,$ a cocycle locally bounded in $t$ with the same property can be obtained by setting
$$\tilde\alpha(t,(y,0))=\min(\alpha(t,(y,0)),h(y)),$$ where $h(y)=N_k^2$ for $y\in A_k\setminus A_{k+1}$ and $$A_k\searrow\emptyset,\ \nu\Bigl(\bigcup_{n=0}^{N_k}S^{-n}A_k\Bigr)>1-\varepsilon_k,\  \varepsilon_k\searrow0.$$
\end{rem}

\begin{ex} Condition \eqref{int2} is not necessary for convergence $\mu$-a.e. of ratios $\a(t,x)/t$ as $t\to\infty$ even for ergodic flows. As an example, we can take a suspension flow built over an ergodic basic transformation $S\colon[0,1]\to[0,1]$ preserving Lebesgue measure, under the function $f(y)=y^{-2/3},$ and define the cocycle by the formula $$\alpha(t,(y,0))=\begin{cases}\sqrt{n},& t=n\in\mathbb N\cap[0,f(y)],\\ 0,& t\in\mathbb N^c\cap[0,f(y)].\end{cases}$$
\end{ex}

Although the means of an integrable cocycle may not converge, it turns out that convergence in density takes place.

The {\it upper density} of the Borel set $\tau\subset\mathbb R_+$ is the limit $$\d(\tau)=\varlimsup_{t\to\infty}\frac{\lambda(\tau\cap[0,t])}{t},$$
where $\lambda$ is Lebesque measure.
The function $f(t)$ {\it converges in density} to $l$ as $t\to\infty$ (we will write $\dlim\limits_{t\to\infty}f(t)=l$) if there exists a set $\tau\subset\mathbb R_+$ of density 0 such that $$\lim\limits_{t\to\infty,\ t\notin\tau}f(t)=l,$$ or equivalently $$\forall\varepsilon>0\ \ \d\{t>0:|f(t)-l|\ge\varepsilon\}=0.$$
In our case, the set $\tau$ will have a finite Lebesgue measure.
\par Below in this section, we will consider measure-preserving semiflows~$\{T^t\}_{t\ge0}.$

\begin{theorem}\label{add} If $\a(t,x)\in L^1(X,\mu)$ for each $t,$ then almost everywhere there exists the limit $$\dlim\limits_{t\to\infty}\frac{\a(t,x)}{t}=\beta(x),$$ where the function $\beta$ is measurable, $T^t$-invariant, and $$\int\beta(x)\,d\mu(x)=1/t\int\a(t,x)\,d\mu(x).$$ Moreover, the neglected subset of the time axis can be chosen measurably depending on $x$ and having a finite Lebesgue measure for each $x$.
\end{theorem}

\begin{proof}
 Since $$\alpha(t,x)=\sum_{n=0}^{[t]-1}\alpha(1,T^nx)+\alpha(\{t\},T^{[t]}x),$$ by the Birkhoff theorem for an automorphism, this convergence is equivalent to convergence of $\alpha(\{t\},T^{[t]}x)/t.$

We have $$\int|\alpha(t+s,x)|\,d\mu(x)\le\int|\alpha(t,x)|\,d\mu(x)+\int|\alpha(s,T^tx)|\,d\mu(x)=$$ $$=\int|\alpha(t,x)|\,d\mu(x)+\int|\alpha(s,x)|\,d\mu(x).$$
Therefore, the measurable, subadditive function $t\mapsto\int|\alpha(t,x)|\,d\mu(x)$ is locally bounded (\cite{K}, p. 461).

Denote $$\Delta_n^{\varepsilon}(x):=\left\{t\in[n,n+1): \frac{|\a(\{t\},T^nx)|}{n}\ge\varepsilon\right\}.$$
Note that the series with positive terms $\sum_n\lambda(\Delta_n^{\varepsilon}(x))$ converges almost everywhere since the series of integrals converges:
$$\sum_n\int_X\lambda(\Delta_n^{\varepsilon}(x))\,d\mu(x)=\sum_n\int_X\int_0^1I_{\{|\a(t,x)|\ge\varepsilon n\}}\,dt\,d\mu(x)=$$
$$=\sum_n\int_0^1\mu\{|\a(t,x)|\ge\varepsilon n\}\,dt\le\frac{1}{\varepsilon}\int_0^1\int_X|\a(t,x)|\,d\mu(x) dt\le$$ $$\le\frac{1}{\varepsilon}\sup\limits_{0\le t\le1}\int\limits_X|\alpha(t,x)|\,d\mu(x)<\infty.$$
Let $\tau_n(x)=\cup_k\Delta_k^{\varepsilon_n}(x),$ where $\varepsilon_n$ is some sequence decreasing to zero and
$$k_n(x):=\min\left\{k: \lambda(\tau_n(x)\cap[k,\infty))\le\frac{1}{2^n}\right\}.$$ Then, as the required set of finite measure, on whose complement convergence holds, we can take
$$\tau(x)=\bigcup_n\{\tau_n(x)\cap[k_{n-1}(x),k_n(x))\}.$$
\end{proof}

\begin{cor} Under the condition $\a(t,x)\in L^1(X,\mu)$ (for all $t$) the lattice limit $\lim\limits_{n\to\infty}\a(nh,x)/nh$ does not depend on the lattice spacing $h.$
\end{cor}

\begin{remark} Obviously, the discarded subset of the $t$ axis in Theorem \ref{add} can be chosen of arbitrary small measure.
\end{remark}

\begin{remark} Unlike the case of absolutely continuous cocycles, there is no local ergodic theorem for arbitrary measurable cocycles, as we see from the example of the Brownian motion cocycle: the process $B^t/t=sB^{1/s} = \widetilde B^s $ is itself a Brownian motion and does not converge as $s\to\infty$ in any sense.
\end{remark}

As a corollary of Theorem \ref{add}, we obtain one else asymptotic property.

\begin{theorem} Let $\{T^t\}$ be an ergodic semiflow, $\a(t,x)\in L^1(X,\mu)$ for each~$t,$ and $\int\a(t,x)\,d\mu(x)=0$ (if this is true for one $t,$ then for all). Then
$$\lambda\{t>0: |\a(t,x)|\le\varepsilon\}=\infty.$$
\end{theorem}

This statement generalizes the well-known fact about the recurrence of random walks on the line and its ergodic analogue, the Atkinson theorem \cite{A}, and is proved in the same way as in the absolutely continuous case \cite{Sh}.

We also give some more explicit than in Theorem \ref{add} constructions of suitable sets $\tau(x)$ of density 0, which, however, can have infinite measure.
Put
$$\tau(x)=\bigcup_n\Delta_n(x),\ \Delta_n(x):=\left\{t\in[n,n+1): |\a(\{t\},T^nx)|\ge \frac{n}{\varphi(n)}\right\},$$
where $\varphi$ is a monotone function increasing to $\infty$. Obviously, we have
\begin{equation}\frac{\a(\{t\},T^{[t]}x)}{t}\to0,\ t\to\infty,\ t\notin\tau(x).\label{lim}
\end{equation}
Let us estimate the growth of the measure $\lambda(\tau(x)\cap[0,t])$. Note that the cocycle $\a(t, x)$ is locally integrable over $t$ for almost all $x$ since, by the Tonelli theorem, $$\int\limits_X\int\limits_0^t|\alpha(s,x)|\,ds\,d\mu(x)=\int\limits_0^t\int\limits_X|\alpha(s,x)|\,d\mu(x)ds\le t\sup\limits_{0\le s\le t}\int\limits_X|\alpha(s,x)|\,d\mu(x)<\infty$$
due to the local boundedness of the function $t\mapsto\int|\alpha(t,x)|\,d\mu(x).$
It follows from the Birkhoff theorem for discrete time that there is a constant~$C,$ depending on $x,$ such that for each $n$
$$C(x)n\ge\sum_{k=1}^{n}\int_0^1|\a(s,T^kx)|\,ds\ge\sum_{k=1}^{n}\frac{k}{\varphi(k)}\lambda(\Delta_k(x)).$$
Applying the Abel transform to the last expression, for $S_n(x)=\sum_{k=1}^n\lambda(\Delta_k(x)),$ we get
$$S_n(x)\le\varphi(n)\Bigl(C(x)+\frac{1}{n}\sum_{k=1}^n\Bigl(\frac{k}{\varphi(k)}-\frac{k-1}{\varphi(k-1)}\Bigr)S_{k-1}(x)\Bigr).$$
Hence by induction,
$$S_n(x)\le C(x)\sum_{k=1}^n\frac{\varphi(k)}{k}=O(\varphi(n)\ln n).$$
We thus obtain an ``almost logarithmic'' estimate for the growth of the set $\tau$: $$\lambda(\tau(x)\cap[0,t])=O(\varphi(t)\ln t)).$$
The same is true for the set $\tau$ constructed from $\Delta_n$ of the form
$$\Delta_n(x)=\left\{t\in[n,n+1): |\a(\{t\},T^nx)|\ge \frac{n}{\varphi\bigl(\frac{f(T^nx)}{n}\bigr)}\right\},$$
$$f(x)=\int_0^1|\a(s,x)|\,ds,$$
 where $\varphi$ is a convex function such that $\lim_{x\to0+}\varphi(x)=\infty$ and $\lim_{x\to0+}x\varphi(x)=0.$ E.g., $\varphi(x)=x^{-1/2}.$\label{phi}

Since, by the Borel-Cantelli lemma, $f(T^nx)/n\to0,\ n\to\infty,$ almost everywhere,
it follows that~\eqref{lim} holds. Also,
$$\lambda(\Delta_n(x))\le\frac{f(T^nx)}{n}\varphi\Bigl(\frac{f(T^nx)}{n}\Bigr)\to0,\ n\to\infty,$$
which implies that the set $\tau$ has density 0. By Jensen's inequality, $$\lambda(\tau(x)\cap[0,n+1])\le\sum_{k=1}^{n}\frac{f(T^kx)}{k}\varphi\Bigl(\frac{f(T^kx)}{k}\Bigr)\le\Bigl(\sum_{k=1}^{n}\frac{f(T^kx)}{k}\Bigr)\varphi\Bigl(\frac{1}{n}\sum_{k=1}^{n}\frac{f(T^kx)}{k}\Bigr).$$
Using the Abel transform and the Birkhoff theorem, one can obtain the asymptotics
$$\sum_{k=1}^{n}\frac{f(T^kx)}{k}\asymp \ln n,$$ from which it follows that the above estimate is almost logarithmic.

Note that an analogue of Kingman's subadditive ergodic theorem \cite{Ki} is also valid.

\begin{theorem}\label{subadd} Let $\a\colon \mathbb R_+\times X\to\mathbb R\cup\{-\infty\}$ be a measurable function such that $\a^+(t,x)\in L^1(X,\mu)$ for each $t$ and $$\a(t+s,x)\le\a(t,x)+\a(s,T^tx)$$ for all $s,t\in\mathbb R_+,$ $x\in X.$ Then there exists $T^t$-invariant function $\beta\colon X\to\mathbb R\cup\{-\infty\}$ such that $\beta^+\in L^1(X,\mu),$
$$\lim\limits_{t\to\infty}\frac{1}{t}\int\a(t,x)\,d\mu(x)=\inf\limits_{t}\frac{1}{t}\int\a(t,x)\,d\mu(x)=\int\beta(x)\,d\mu(x),$$
and for $\mu$-a.e. $x\in X,$ there exists the limit
$$\dlim\limits_{t\to\infty}\frac{\a(t,x)}{t}=\beta(x)$$
along the complements to subsets of the time axis of finite Lebesgue measure.
\end{theorem}

\begin{proof} We have
\begin{equation}
\a([t]+1,x)-\a(1-\{t\},T^tx)\le\a(t,x)\le\a([t],x)+\alpha(\{t\},T^{[t]}x).\label{1}
\end{equation}
Taking into account that the subadditive function $t\mapsto\int\alpha^+(t,x)\,d\mu(x)$ is locally bounded, applying the arguments of Theorem \ref{add} to the sets
$$\Delta_{n,1}(x)=\bigl\{t\in[n,n+1):\a^+(\{t\},T^nx)\ge\varepsilon n\bigr\},$$
$$\Delta_{n,2}(x)=\bigl\{t\in[n,n+1):\a^+(1-\{t\},T^tx)\ge\varepsilon n\bigr\},$$
we find sets $\tau_{1,2}(x)$ of finite measure, on whose complements there is convergence
$$\frac{\a^+(\{t\},T^{[t]}x)}{t}\to0,\ \frac{\a^+(1-\{t\},T^tx)}{t}\to0$$
respectively. As the required set of density 0, we can take $\tau_1(x)\cup\tau_2(x).$ Indeed, by the subadditive ergodic theorem for discrete time, there exists the limit $\lim_{t\to\infty}\a([t],x)/t=:\beta(x)$ with $\beta^+\in L^1(X,\mu).$ Therefore,
$$\varlimsup\limits_{t\to\infty,\,t\notin\tau_1\cup\tau_2}\frac{\a(t,x)}{t}\le\varlimsup\limits_{t\to\infty,\,t\notin\tau_1}\frac{\a(t,x)}{t}\le\beta(x),$$ $$\varliminf\limits_{t\to\infty,\,t\notin\tau_1\cup\tau_2}\frac{\a(t,x)}{t}\ge\varliminf\limits_{t\to\infty,\,t\notin\tau_2}\frac{\a(t,x)}{t}\ge\beta(x).$$
 Since the function $t\mapsto\int\a(t,x)\,d\mu(x)$ is subadditive, there exists the limit $$\lim\limits_{t\to\infty}\frac{1}{t}\int\a(t,x)\,d\mu(x)=\inf\limits_{t}\frac{1}{t}\int\a(t,x)\,d\mu(x).$$ And from $\eqref{1}$ and the fact that the function $t\mapsto\int\a^+(t,x)d\mu(x)$ is locally bounded it follows that $$\int\beta(x)\,d\mu(x)=\lim\limits_{t\to\infty}\frac{1}{t}\int\a([t]+1,x)\,d\mu(x)\le\lim\limits_{t\to\infty}\frac{1}{t}\int\a(t,x)\,d\mu(x)\le$$
$$\le\lim\limits_{t\to\infty}\frac{1}{t}\int\a([t],x)\,d\mu(x)=\int\beta(x)\,d\mu(x).$$
\end{proof}

\section {Multiplicative ergodic theorem}

\medskip

The following theorem shows that the structure of the Oseledets invariant subspaces is preserved under our weaker integrability conditions for cocycles.

\begin{theorem}[multiplicative ergodic theorem]
Let $\A\colon\mathbb R\times X\to GL(m,\mathbb R)$ be a cocycle with $\ln^+\|\A(t,x)^{\pm1}\|\in L^1(X,\mu)$ for each $t\in\mathbb R$. Then for almost every~$x$
\begin{itemize}
\item[\textup{(i)}] there exists the limit $$\dlim\limits_{t\to\infty}(\A^*(t,x)\A(t,x))^{\frac{1}{2t}}=:\Lambda(x);$$
\item[\textup{(ii)}] there exists a measurable splitting $\mathbb R^m=\bigoplus\limits_{i=1}^{k(x)}U_i(x)$ such that $$U_i(T^tx)=\A(t,x)U_i(x)$$ and $$\dlim\limits_{t\to\pm\infty}\frac{1}{|t|}\ln\|\A(t,x)v\|=\pm\chi_i(x),\ v\in U_i(x)\setminus\{0\}$$ uniformly on $U_i(x)\setminus\{0\}.$
    The functions $k(x),$ $\chi_i(x),$ and $\dim U_i(x)$ are $T^t$-invariant and $\exp(\chi_i(x))$ are the eigenvalues of the matrix $\Lambda(x)$ with multiplicities $\dim U_i(x).$
\end{itemize}
Moreover, the convergence holds along the complements to time axis subsets of finite Lebesgue measure (which can be chosen measurably dependent on $x$).
\end{theorem}

This theorem is deduced, as for discrete time, from its one-sided version, in which (ii) is replaced by
{\it
\begin{itemize}
\item[\textup{(ii')}] $$\dlim\limits_{t\to\infty}\frac{1}{t}\ln\|\A(t,x)v\|=\chi_i(x),\ v\in V_i(x)\setminus V_{i+1}(x),$$
    $$V_i(T^tx)=\A(t,x)V_i(x),$$
    $$V_i(x)=\bigoplus\limits_{j=i}^{k(x)}W_j(x),$$ where $W_j(x)$ are the eigenspaces of the operator $\Lambda(x)$ corresponding to its eigenvalues $\exp(\chi_1(x))\ge\ldots\ge\exp(\chi_{k(x)}(x))$.
\end{itemize}}

\medskip

Together with $(i)$ this condition is an analogue of the Lyapunov regularity and, as was noted by V.A. Kaimanovich \cite{K87}, is equivalent to the existence of a positive definite symmetric matrix $\Lambda(x)$ such that
$$\dlim_{t\to\infty}\frac{1}{t}\ln\|(A(t,x)\Lambda^{-t}(x))^{\pm1}\|=0\mbox{ a.e.}$$
If not all $\chi_i(x)$ are equal to 0, then the last condition means the proximity of the trajectory of the inverse cocycle $A(t,x)^{-1}p$ to the geodesics $\gamma(\theta(x)t,x)=\Lambda^{-t}(x)p$ in the symmetric space $GL(m,\mathbb R)/O(m)$ (for which $p=O(m)$) with the corresponding metrics $\rho$ \cite{K87}:
$$\dlim_{t\to\infty}\frac{1}{t}\rho(A(t,x)^{-1}p,\gamma(\theta(x)t,x))=0.$$
(In this case, the trajectory $A(t,x)^{-1}p$ tends to a random point of the boundary at infinity $GL(m,\mathbb R)/O(m)(\infty).$) A more general statement --- a version of the Karlsson-Margulis theorem \cite{KM} --- is also true. We restrict ourselves to the ergodic case.

\begin{theorem} Let $(Y,\rho)$ be a uniformly convex, Busemann nonpositively curved, complete metric space $(Y,\rho);$ and let $A\colon \mathbb R_+\times X\to G$ be an ``inverse'' cocycle, i.e.,
$$A(t+s,x)=A(t,x)A(s,T^tx),$$ over an ergodic semiflow $\{T^t\}$
 with values in semigroup of nonexpanding maps of $Y$. Suppose that for a fixed point $p\in Y,$
$$\rho(A(t,x)p,p)\in L^1(X,\mu)$$ holds for all $t$. Then for almost every $x$ there exists the limit
\begin{equation}\dlim_{t\to\infty}\frac{1}{t}\rho(A(t,x)p,p)=:\theta,\label{v}\end{equation}
 and if $\theta>0,$ then there exists a unique geodesics $\gamma$ in $Y,$ depending on $x,$ with $\gamma(0,x)=p,$ such that
\begin{equation}\dlim_{t\to\infty}\frac{1}{t}\rho(A(t,x)p,\gamma(\theta t,x))=0 \mbox{ a.e.}\label{l}\end{equation}
Moreover, the convergence is fulfilled along the complements to time axis subsets of finite Lebesgue measure.
\end{theorem}
\begin{proof}
Theorem \ref{subadd} implies the existence of the limit \eqref{v}. Statement \eqref{l} follows from the discrete time theorem, the inequalities
$$\rho\bigl(A(t,x)p,\gamma(\theta t)\bigr)\le\rho\bigl(A(t,x)p,A([t],x)\bigr)+\rho\bigl(A([t],x)p,\gamma(\theta[t])\bigr)+
\rho\bigl(\gamma(\theta[t]),\gamma(\theta t)\bigr),$$
$$\rho\bigl(A(t,x)p,A([t],x)p\bigr)=\rho\bigl(A([t],x)A(\{t\},T^{[t]}x)p,A([t],x)p\bigr)\le$$ $$\le\rho\bigl(A(\{t\},T^{[t]}x)p,p\bigr),$$
and the existence of the limit
$$\dlim_{t\to\infty}\frac{1}{t}\rho(A(\{t\},T^{[t]}x)p,p)=0.$$
The latter was in fact already used in applying Theorem \ref{subadd}.
\end{proof}

\begin{rem} The infinite-dimensional operator versions of the MET for convergence in density are also valid. Proofs in the spirit of Raghunathan can be carried out by choosing the naturally arising countable collection of ``bad'' sets of density 0 so that the series of their measures converges.
\end{rem}

\medskip

\noindent Maxim E. Lipatov\\
Dept. of Mechanics and Mathematics\\
Lomonosov Moscow State University \\
Main Building, 1 Leninskiye Gory\\
Moscow 119991\\
RUSSIA

\medskip

\noindent\textit{E-mail:} \texttt{maxim.lipatov@gmail.com}
\end{document}